\documentclass[11pt,oneside,english]{amsart}
\usepackage[T1]{fontenc}
\usepackage[utf8]{inputenc}
\usepackage{color}
\usepackage{babel}
\usepackage{amstext}
\usepackage{amsthm}
\usepackage{amssymb}
\usepackage{graphicx}
\usepackage[caption=false]{subfig}
\usepackage[unicode=true,pdfusetitle,
 bookmarks=true,bookmarksnumbered=false,bookmarksopen=false,
 breaklinks=false,pdfborder={0 0 0},backref=false,colorlinks=true]
 {hyperref}
\hypersetup{
 citecolor=blue}

\makeatletter
\numberwithin{equation}{section}
\numberwithin{figure}{section}
\theoremstyle{plain}
\newtheorem*{cor*}{\protect\corollaryname}
\theoremstyle{plain}
\newtheorem{thm}{\protect\theoremname}[section]
\theoremstyle{definition}
\newtheorem{defn}[thm]{\protect\definitionname}
\theoremstyle{remark}
\newtheorem{rem}[thm]{\protect\remarkname}
\theoremstyle{plain}
\newtheorem{prop}[thm]{\protect\propositionname}
\theoremstyle{plain}
\newtheorem{lem}[thm]{\protect\lemmaname}
\theoremstyle{plain}
\newtheorem{cor}[thm]{\protect\corollaryname}

\usepackage{geometry}
\usepackage{amsmath}

\numberwithin{equation}{section}
\numberwithin{figure}{section}
\usepackage{enumitem}		
 \let\footnote=\endnote
\@ifundefined{lettrine}{\usepackage{lettrine}}{}

\theoremstyle{definition}

\def\s{\sigma}

\def\L{\Lambda}

\def\R{\mathbb{R}}

\def\N{\mathbb{N}}
\def\Z{\mathbb{Z}}

\def\ep{\varepsilon}

\usepackage{float}

\keywords{}

\subjclass[2000]{}

\def\s{\sigma}
\def\Si{\Sigma_\ell}

\def\L{\mathcal{L}}

\def\R{\mathbb{R}}
\def\P{\mathbb{P}}

\def\ep{\varepsilon}
\def\A{\mathcal{A}}
\def\M{\mathcal{M}}

\def\glr{\text{GL}_d(\R)}

\def\hol{H\"older }

\def\span{\text{Span}}
\makeatother

  \providecommand{\corollaryname}{Corollary}
  \providecommand{\definitionname}{Definition}
  \providecommand{\lemmaname}{Lemma}
  \providecommand{\propositionname}{Proposition}
  \providecommand{\remarkname}{Remark}
  \providecommand{\theoremname}{Theorem}
\providecommand{\theoremname}{Theorem}

\usepackage{tikz,xcolor,hyperref}

\definecolor{lime}{HTML}{A6CE39}
\DeclareRobustCommand{\orcidicon}{
	\begin{tikzpicture}
	\draw[lime, fill=lime] (0,0) 
	circle [radius=0.16] 
	node[white] {{\fontfamily{qag}\selectfont \tiny ID}};
	\draw[white, fill=white] (-0.0625,0.095) 
	circle [radius=0.007];
	\end{tikzpicture}
	\hspace{-2mm}
}
\foreach \x in {A, ..., Z}{\expandafter\xdef\csname orcid\x\endcsname{\noexpand\href{https://orcid.org/\csname orcidauthor\x\endcsname}
			{\noexpand\orcidicon}}
}

\author[Reza Mohammadpour, Kiho Park]{Reza Mohammadpour\orcidA{}, Kiho Park}

\address{Department of Mathematics, Uppsala University, Box 480, SE-75106, Uppsala, Sweden}
\date{\today}

\subjclass[2010]{28A80,  37A44,  37D35,  37H15 }
\keywords{Quasi-multiplicativity, matrix cocycles, spannability, Bernoulli Property}%
\email{reza.mohammadpour@math.uu.se}

\address{School of Mathematics, KIAS, 85 Hoegiro, Dongdaemun-gu, Seoul, 02455, Republic of Korea.}
\email{kiho.park12@gmail.com}
\begin{document}
\title[uniform quasi-multiplicativity of locally constant cocycles]{uniform quasi-multiplicativity of locally constant cocycles and applications}
\maketitle

\begin{abstract}
In this paper, we show that a locally constant cocycle $\A$ is $k$-quasi multiplicative under the irreducibility assumption. More precisely, we show that if $\A^t$ and $\A^{\wedge m}$ are irreducible for every $t \mid d$ and $1\leq m \leq d-1$, then $\A$ is $k$-uniformly spannable for some $k\in \N$, which implies that $\A$ is $k$-quasi multiplicative. We apply our results to show that the unique subadditive equilibrium Gibbs state is $\psi$-mixing and calculate the Hausdorff dimension of cylindrical shrinking target and recurrence sets.
\end{abstract}

\section{Introduction and statement of the results}
A \textit{matrix cocycle} $\A$ on a compact metric space $X$ is a continuous map $\A \colon X \to \glr$ over a topological dynamical system $(X, T)$. For $n\in \N$ and $x \in X$, we define the product of $\A$ along the length $n$ orbit of $X$ as   
$$\A^n(x):= \A(T^{n-1}(x)) \ldots \A(x).$$

 The submultiplicativity of the norm $\|\cdot\|$ implies that  $\| \A \|$ is submultiplicative in the sense that for any $m, n \in \mathbb{N}$,
$$
0 \leq \|\mathcal{A}^{n+m}(x)\| \leq \|\mathcal{A}^{m}(T^{n}(x))\| \|\mathcal{A}^{n}(x)\|.
$$
Such submultiplicative sequence gives rise to a norm potential $\Phi_{\mathcal{A}}:=\left\{\log \| \A^{n}\| \right\}_{n \in \mathbb{N}}$.

Let $\ell\in \N$ be given.
The one-sided shift $\Si$ of $\ell$ symbols is a space $\{1,2,\ldots,\ell\}^\N$. Let $\sigma: \Si \to \Si$ be the left shift map defined by $\sigma i=i_{1} i_{2} \cdots$ for all $i=i_{0} i_{1} \cdots \in \Si$ and for simplicity, we denote it by $(\Si, \s)$. We will focus on locally constant cocycles $\A$, which are matrix cocycles $\A: \Si \to \glr$ over a one-sided shift $(\Si, \sigma)$ that depends only on the zero-th symbol $x_0$ of $x = (x_i)_{i\in \N}$. Assume that $(A_1, \ldots, A_{\ell})\in \glr^{\ell}$ generated a locally constant cocycle $\A \colon \Si \to \glr$. We say that $\A \colon \Si \to \glr$ is \textit{irreducible} if there does not exist a proper subspace $V \subset \R^d$ preserved such that $A_i V \subset V$ for $i=1, \ldots, \ell$.  We also say that $\A \colon \Si \to \glr$ is \textit{strongly irreducible} if there does not exist a finite collection $V_{1}, \ldots, V_{m}$ of non-zero proper subspaces $V_{j}$ such that $A_{i}\left(\bigcup_{j=1}^{m} V_{j}\right)=\bigcup_{j=1}^{m} V_{j}$ for every $i=1, \ldots, \ell$. 

For any length $n$ word $I=i_{0}\ldots  i_{n-1}$ (see Section \ref{Preliminary} for the definition), we denote
\[ \mathcal{A}_{I}:=A_{i_{n-1}}\ldots A_{i_{0}}.\]

Denoting by $\L$ the set of all finite words, we say a locally constant cocycle $\A: \Si \to \glr $ generated by $(A_1, \ldots, A_{\ell})$ is \textit{quasi-multiplicative} if there exist $k\in \N$ and $c>0$ such that for any $I,J \in \L$, there exists $K = K(I,J) \in \L$ with $|K| \leq k$ such that
\[
\|\A_{IKJ}\| \geq c\|\A_{I}\|\|\A_{J}\|.
\]

Notice that quasi-multiplicativity of $\A$ resembles Bowen’s specification property \cite{Bow74} in some respects. Feng \cite{feng2003lyapunov} showed that the quasi-multiplicativity property implies the uniqueness of the Gibbs equilibrium measure for the norm of the cocycle $\A.$ Feng \cite{feng09} also proved that if $\A$ is a locally constant $\glr$-cocycle over a full
shift generated by an irreducible set of matrices, then $\A$ is quasi-multiplicative. Recently, there has been further study in quasi-multiplicativity (see \cite{ KM18, park2020quasi, barany2021dynamically, Morris21} for instance).
Unfortunately, the
lack of control on the length of the connecting word $K$ from quasi-multiplicativity is a limitation in studying important applications such as Bernoulli property on a class of subadditive equilibrium
states and shrinking target and recurrence sets; see Section \ref{Applications}.  When the connecting word $K$ in the quasi-multiplicativity property has a fixed length $k\in \N$, we say $\A$ is $k$-\textit{quasi-multiplicative}; see Definition \ref{Def-k-qm}. 

There are a few results along this line in the literature.  In the same setting of locally constant cocycles,  B\'ar\'any and Troscheit \cite[Proposition 2.5]{barany2021dynamically} and Morris \cite[Theorem 7]{Morris21} proved that $\A$ is $k$-quasi-multiplicative when $\A$ is (strongly) irreducible and proximal. Note that if $\A$ is strongly irreducible, then $\A^t$ is irreducible for every $1\leq t \leq d.$ In this paper, we generalize their results. 
A distinction in our result from similar results is that we only require versions of irreducibility as our assumptions, while many recent similar results additionally need some form of proximality to obtain $k$-quasi-multiplicativity property. 
In fact, our result is inspired by a recent result of Bochi and Garibaldi \cite[Proposition 3.9]{bochi2019extremal} in a more general setting; see Remark \ref{remark-qm} for further comments on their work.

We say that a locally constant cocycle $\A \colon \Si \to \glr$ is \textit{$k$-uniformly spannable} (See Subsection \ref{subsection-span} for more information) if there exists $k \in \N$ such that for any nonzero vector $u \in \R^d$,  \[V_{u,k} = \R^d,\]
where $V_{u,k} := \span\{\A_Iu \colon I \in \L \text{ with } |I|=k \} \subseteq \R^d$.

We will also make use of the exterior product cocycle $\A^{\wedge m}$ for $1 \leq m \leq d-1$ where $\A^{\wedge m}(x)$ is considered as a linear transformation on $(\R^d)^{\wedge m}$. Our main result is as follows:

\begin{thm}\label{thm: main} Let $\A\colon \Si \to \glr$ be a locally constant cocycle. Suppose $\A^t$ and $\A^{\wedge m}$ are irreducible for every $t \mid d$ and $1\leq m \leq d-1$. Then $\A$ is $k$-uniformly spannable for some $k\in \N$.
\end{thm}

 As spannable cocycles are quasi-multiplicative,  the following corollary is immediate of the above theorem.

\begin{cor}\label{thm-k-qm}
Let $\A\colon \Si \to \glr$ be a locally constant cocycle. Suppose $\A^t$ and $\A^{\wedge m}$ are irreducible for every $t \mid d$ and $1\leq m \leq d-1$. Then $\A$ is $k$-quasi-multiplicative for some $k\in \N$.
\end{cor}

Corollary \ref{thm-k-qm} has nice applications in subadditive thermodynamic formalism and number theory, which we discuss in more detail in Section \ref{Applications}.

The paper is organized as follows. In Section \ref{Preliminary}, we introduce relevant notations and prove Corollary \ref{thm-k-qm}. In Section \ref{proof of the main thm}, we prove Theorem \ref{thm: main}, and in Section \ref{Applications}, we show some applications of Corollary \ref{thm-k-qm}.
\subsection{Acknowledgements.}
The authors would like to express their gratitude to the anonymous referee for their valuable corrections and suggestions, which greatly contributed to the improvement of the paper. R. Mohammadpour was supported by the Knut and Alice Wallenberg Foundation. 

\section{Preliminary}\label{Preliminary}
\subsection{Set up}

For each $n\in\N$, we define $\L(n)$ to be the set of all length $n$ words of $\Si$, and we define $\L:=\bigcup\limits_{n\in \N}\L(n)$ to be the set of all words. If $i=i_{0} i_{1} \cdots \in \L$, then we define $\left.i\right|_{n}=i_{0} \cdots i_{n-1}$ for all $n \in \mathbb{N}$. The empty word $\left.i\right|_{0}$ is denoted by $\varnothing$. The length of $i \in \L$ is denoted by $|i|$. The longest common prefix of $i, j \in \L \cup \Si $ is denoted by $i \wedge j$. The concatenation of two words $i \in \L \cup \Si$ and $j \in \L $ is denoted by $ji$.  If $i \in \L(n)$ for some $n$, then we set $[i]=\left\{j \in \Si:\left.j\right|_{n}=i\right\}$. The set $[i]$ is called a \textit{cylinder set}. A cylinder containing $x=\left(x_{i}\right)_{i \in \mathbb{Z}} \in \Si$ of length $n \in \mathbb{N}$ is defined by
$[x]_{n}:=\left\{\left(y_{i}\right)_{i \in \mathbb{N}} \in \Si : x_{i}=y_{i} \text { for all } 0 \leq i \leq n-1 \right\}.$  The shift space $\Si$ is compact in the topology
generated by the cylinder sets. Moreover, the cylinder sets are open and closed in
this topology and they generate the Borel $\sigma$-algebra. We denote by $\M(\Si, \sigma)$ the space of all $\sigma$-invariant Borel probability measures
on $\Si$.

\subsection{Quasi-multiplicativity and Spannability}\label{subsection-span}

\begin{defn}We say a locally constant cocycle $\A \colon \Si \to \glr$ is \textit{quasi-multiplicative} if there exist $k\in \N$ and $c>0$ such that for any $I,J \in \L$, there exists $K = K(I,J) \in \L$ with $|K| \leq k$ such that
\begin{equation}\label{eq: QM}
\|\A_{IKJ}\| \geq c\|\A_{I}\|\|\A_{J}\|.
\end{equation}
\end{defn}

\begin{defn}\label{Def-k-qm}
We say a locally constant cocycle $\A \colon \Si \to \glr$ is \textit{$k$-quasi-multiplicative} for some $k\in \N$ if there exists $c>0$ such that for any $I,J \in \L$, there exists $K = K(I,J) \in \L(k)$ such that \eqref{eq: QM} holds.
\end{defn}

We will elaborate more on the applications of $k$-quasi-multiplicativity in Section \ref{Applications}. In the remaining part of this subsection, we describe a notion of spannability, that is closely related to quasi-multiplicativity. 
In what follows, we will repeatedly make use of the following notation: given a locally constant cocycle $\A$, a vector $u \in \R^d$ and an integer $k\in \N$, we define 
\begin{equation}\label{eq: V_uk}
V_{u,k} := \span\{\A_Iu \colon I \in \L(k)\} \subseteq \R^d.
\end{equation}

\begin{defn}We say a locally constant cocycle $\A \colon \Si\to \glr$ is \textit{spannable} if for any nonzero vector $u \in \R^d$, there exists $k = k(u) \in \N$ such that $V_{u,k} = \R^d$.
If $k = k(u)$ can be chosen uniformly in $u$, then we say $\A$ is \textit{$k$-uniformly spannable}.
\end{defn}
\begin{rem}
We note that if $V_{u,k}$ is equal to the entire subspace $\R^d$ for some $u\in \R^d$ and $k \in \N$, from continuity so does $V_{v,k}$ for all $v \in \R^d$ in a small neighborhood of $u$.
Moreover, if $V_{u,k}$ is equal to $\R^d$, then so does $V_{u,k+1}$. In particular, if $\A$ is $k$-uniformly spannable for some $k$, then it is $k'$-uniformly spannable for any $k' \geq k$.
\end{rem}

Throughout the paper, when we measure the angle between nonzero vectors, we mean the angle between the lines spanned by the vectors.  Similarly, when we measure an angle between a nonzero vector $v$ and a hyperplane $\mathbb{W}$, we mean the minimum angle $\measuredangle(v, w)$ over all $w \in \mathbb{W} \backslash\{0\}$.  The following statement can be found in \cite[Proposition 8]{butler2021thermodynamic} which states that spannability implies quasi-multiplicativity.

\begin{prop}\label{spannable-QM} Suppose a locally constant cocycle $\A \colon \Si\to\glr$ is $k$-uniformly spannable, then $\A$ is $k$-quasi-multiplicativite.
\end{prop}
\begin{proof} 
The proof is similar to \cite[Proposition 8]{butler2021thermodynamic} for fiber-bunched cocycles. We give a sketch of the proof here for
the convenience of the readers.

For any $A \in \glr$, let $v_{A,1}\in \R^d$ be a unit vector such that $\|A v_{A,1}\| = \|A\|$, and let $v_{A,2}$ be the unit vector in the direction of $A v_{A,1}$. 

From $k$-uniform spannability,  we begin by finding $\ep>0$ such that for given arbitrary $I,J \in \L$,  we can find $K\in \L(k)$ such that $\measuredangle(\A_Kv_{\A_I,2}, (v_{\A_J,1})^\perp ) \geq \ep$. By \cite[Lemma 4.5]{park2020quasi}, that translates to the existence of $c>0$ satisfying the $k$-quasi-multiplicativity \eqref{eq: QM}.
\end{proof}

\begin{proof}[Proof of Corollary \ref{thm-k-qm}]
It follows from the combination of Theorem \ref{thm: main} and Proposition \ref{spannable-QM}.
\end{proof}

We end this subsection by commenting on a class of cocycles that generalize the class of locally constant cocycles and by comparing how the notions defined above relates to such cocycles.
\begin{rem}\label{remark-qm} Beyond locally constant cocycles, there exists a subset of \hol continuous cocycles that are nearly conformal. We call them \textit{fiber-bunched cocycles}. The most useful property of fiber-bunched cocycles is the existence of holonomies and often we tend to think them as generalizations of locally constant cocycles; see \cite{bonatti2004lyapunov, park2020quasi, Moh22}. 

The properties introduced above, such as quasi-multiplicativity and spannability, can be successfully extended, and sufficient conditions have been found which imply such properties. For instance, above mentioned result of Feng \cite{feng2003lyapunov} on the uniqueness of the equilibrium state using quasi-multiplicativity remains to hold for fiber-bunched cocycles. Moreover, Park \cite{park2020quasi} showed that typicality, an assumption introduced by Bonatti and Viana \cite{bonatti2004lyapunov} to replicate the effect of strong irreducibility and non-compactness from the classical work of Furstenberg \cite{furstenberg1963noncommuting} implies quasi-multiplicativity. The $k$-uniform spannability introduced above was motivated by the work of Bochi and Garibaldi \cite{bochi2019extremal} who showed that irreducibility, along with an extra assumption on how close the cocycle is from being conformal, implies uniform spannability. They use the term \textit{uniform spannability}, when translated to our setting of locally constant cocycles, to roughly mean $\bigcup\limits_{k=1}^nV_{u,k} = \R^d$ for some $n \in \N$. In order to distinguish from their version of uniform spannability, we have decided to use \textit{$k$-spannability} to denote the stronger statement that $V_{u,k}=\R^d$.
\end{rem}

\section{Proof of Theorem \ref{thm: main}}\label{proof of the main thm}
In this section, we will prove Theorem \ref{thm: main}. Whenever we write $V = W$ for two $m$-dimensional subspaces of $\R^d$, we mean they are equal considered as elements of the Grassmannian $Gr(m,\R^d)$.
Moreover, for $I \in \L$ we define 
$$\A_IV:=\span\{\A_Iv \colon v \in V\}.$$

\begin{proof}[Proof of Theorem \ref{thm: main}]
Suppose on the contrary that there does not exist $k\in \N$ such that $\A$ is $k$-uniformly spannable, meaning that for every $k\in \N$ there exists $u=u_k \in \R^d$ such that $V_{u,k}$ is a proper subspace of $\R^d$.
Defining an open set
$$S_k:=\{u \in \R^d \colon V_{u,k} = \R^d\},$$ this means that its complement $T_k:=\R^d \setminus S_k$ is a closed non-empty set for every $k\in \N$. Moreover, it is clear from the definition that $S_k \subseteq S_{k+1}$, meaning that $T_k$ satisfies the reverse inclusion $T_{k+1} \subseteq T_{k}$. 
Therefore, the nested intersection $\displaystyle \bigcap\limits_{k \in \N} \P T_{k}$ is necessarily non-empty. In particular, we can choose a vector $u\in \R^d$ which belongs to $T_k$, meaning that $V_k:=V_{u,k}$ is a proper subspace of $\R^d$, for every $k\in \N$.

As $n \mapsto \dim(V_n)$ is a bounded non-decreasing function, the dimension of $V_k$ has to stabilize to some $m \in \N$ strictly smaller than $d$. By dropping the first few subspaces from the sequence $\{V_k\}_{k\in \N}$ we may assume that the dimension of $V_k$ is equal to $\gamma$ for all $k \in \N$. Moreover, from the definition \eqref{eq: V_uk} of $V_k$, we have
\begin{equation}\label{eq: V relation}
V_{k+n} = \A_IV_k = \A_JV_k
\end{equation}
for every $k,n\in \N$ and $I,J \in \L(n)$.
This is the defining characteristic of the sequence $\{V_k\}_{k\in \N}$. Now, by choosing a possibly different sequence of subspaces we may assume that the common dimension $\gamma$ of $\{V_k\}_{k\in \N}$ is the smallest such number; that is, if $\{W_k\}_{k\in \N}$ is another sequence of subspaces of common dimension satisfying \eqref{eq: V relation}, then its dimension is at least $\gamma$.

\begin{lem} For any $i \neq j$, the subspaces $V_i$ and $V_j$ either coincide or intersect trivially.
\end{lem}

\begin{proof}
Suppose there exists $i \neq j$ such that $W:=V_i \cap V_j$ is a non-trivial proper subspace of both $V_i$ and $V_j$. Then for any $n\in \N$ and $I \in \L(n)$ we have 
$$\A_IW = \A_I(V_i \cap V_j) = V_{i+n} \cap V_{j+n},$$
where the resulting subspace $V_{i+n} \cap V_{j+n}$ does not depend on $I$. In particular, this allows us to define a sequence of subspace $\{W_k\}_{k \in \N}$ satisfying \eqref{eq: V relation} where $W_k:=\A_IW$ for any $I \in \L(k)$, but this is a contradiction to the choice of $\{V_k\}_{k \in \N}$ being the sequence of subspaces of the smallest dimension satisfying \eqref{eq: V relation}.
\end{proof}

Following the lemma, we will now consider two separate cases, and conclude that both cases will lead to a contradiction, and hence deduce that $\A$ must be $k$-uniformly spannable for some $k \in \N$.\\

\noindent \textbf{Case 1}: The first case is when $V_{i}= V_{j}$ for some distinct $i,j\in \N$. Since every $V_k$ has the same dimension, the sequence $\{V_k\}_{k\in \N}$ must be periodic with some period $t\in \N$; that is, $t\in \N$ is a smallest integer such that $V_1 = V_{t+1}$. Moreover, $V_i$ and $V_j$ have a trivial intersection for distinct $1 \leq i,j \leq t$. Recalling that each $V_k$ has dimension $\gamma$, the subspace $W:=\span\{v \in V_i \colon i = 1,2,\ldots, t\}$ is a $\gamma t$-dimensional subspace of $\R^d$ preserved under $\A$. Since $\A$ is irreducible from the assumption, this implies that $W$ must be the entire subspace $\R^d$, and that $t$ divides $d$. However, this implies that $\A^t$, which preserves a nontrivial subspace $V_1$ (or any one of $V_k$), is reducible; a contradiction. \\
 
\noindent \textbf{Case 2}: The second case is any two subspaces from $\{V_k\}_{k \in\N}$ have a trivial intersection for any distinct $i,j \in \N$, which is the choice $m=\gamma$. We will show that we also arrive at a contradiction in this case.

We begin by choosing a decomposable vector $w_k = v_{1,k} \wedge \ldots \wedge v_{m,k} \in (\R^d)^{\wedge m}$ where $\span\{v_{1,k} , \ldots , v_{m,k}\}$ coincides with $V_k$. Since $A_iV_k = A_jV_k = V_{k+1}$ from \eqref{eq: V relation} for any $A_i,A_j$ in the image of $\A$, decomposable vectors $A_i^{\wedge m}w_k$ and $w_{k+1}$ differ by a multiplicative constant.
Moreover, each $w_k$ is an eigenvector of $B_{i,j}:=(A_i^{\wedge m})^{-1}A_j^{\wedge m}$. We now fix any $i \neq j$ such that $B:=B_{i,j}$ is not a scalar multiple of the identity transformation.  Such choice of $i,j$ is possible because otherwise every $A_i^{\wedge m}$ would be a scalar multiple of one another, which contradicts the assumption that $\A^{\wedge m}$ is irreducible for all $1\leq m\leq d-1$. Now let $\lambda_k$ be the corresponding eigenvalue of $w_k$.

Considered as a subspace of $(\R^d)^{\wedge m}$, let $W_k:=\span\{w_1,\ldots,w_k\}$ for each $k \in \N$. We claim that the subspace $W_N$ where $N:={d \choose m}$ is equal to the entire $(\R^d)^{\wedge m}$ because otherwise there would exist some $k < N$ such that $W_{k} = W_{k+1}$, meaning that $w_{k+1}$ belongs to $ W_{k}$. However, this would imply that $\A^{\wedge m}$ preserves a proper subspace $W_k$ of $(\R^d)^{\wedge m}$ because we can inductively show that $w_l$ (which is a scalar multiple of $A_i^{\wedge m}w_{l-1}$ for any $A_i$) belongs to $W_{k}$ for every $l>k$, and this contradicts the assumption that $\A^{\wedge m}$ is irreducible for all $1 \leq m\leq d-1$. 

Therefore, $W_N$ coincides with $(\R^d)^{\wedge m}$, and $\{w_1,\ldots, w_N\}$ forms a basis of $(\R^d)^{\wedge m}$. Now choose for any $l>N$, so that $w_l$ can be written as
$$w_l = c_1w_1 +\ldots + c_Nw_N.$$

Applying $B$ on both sides and using the fact that each $w_k$ is an eigenvector of $B$ with eigenvalue $\lambda_k$, we get $\lambda_lw_l = \sum\limits_{i=1}^N c_i\lambda_iw_i$. Substituting the above equation for $w_l$ into $\lambda_lw_l$ and equating coefficients gives $\lambda_l = \lambda_1 = \ldots =\lambda_N$, and this means that $B$ must be a scalar multiple of the identity transformation. However, it is a contradiction to the choice of $B$.\\

Since both cases lead to a contradiction, this completes the proof.
\end{proof}

\section{Applications}\label{Applications}
This section is devoted to showing some applications of our main theorems.
\subsection{Gibbs matrix equilibrium states have the Bernoulli property}

For any matrix cocycle $\A: X \to \glr$ over a topological dynamical system $(X, T)$, and $s\geq 0$, the potential $\Phi_{\A}^{s}:=\{s \log \|\A^n\|\}_{n \in \N}$ is subadditive. Therefore, the theory of subadditive thermodynamic formalism applies. For instance, the subadditive variational principle (see \cite{CFH08}) states that
$$
P\left(\Phi_{\mathcal{A}}^{s}\right)=\sup _{\mu \in \mathcal{M}(X, T)}\left\{h_{\mu}(T)+s \lim _{n \rightarrow \infty} \frac{1}{n} \int \log \|\A^n(x)\| d \mu\right\}.
$$
Any invariant measures achieving the supremum are called the \textit{equilibrium states} of $\Phi_{\mathcal{A}}^{s}$. If the entropy map is upper semi-continuous, the supremum is always attained. The thermodynamic interpretation of the parameter $s$ is as an inverse temperature of a system (see \cite{ru}), while the equilibrium measure $\mu_{s}$ of  $\Phi_{\mathcal{A}}^{s}$ describes the equilibrium of the system at temperature $1 / s$. The $s \rightarrow \infty$ limit is therefore a \textit{zero temperature limit}, and an accumulation point of the $\mu_{s}$ can be interpreted as a ground state (see, e.g., \cite{C, Je,  Jenkinson-MU, Bremont, Morris-zero, Moh20a}).

Feng and K\"aenm\"aki \cite{FK11} showed that if $\A$ is a locally constant $\glr$-cocycle over a full
shift generated by an irreducible set of matrices, then $\Phi_{\A}^{s}$
 has a unique equilibrium state $\mu_s$ for all $s\geq 0$. Moreover, $\mu_s$ has the subadditive
\textit{Gibbs property}: there exists $C_0>1$ such that for any $x\in \Si$ and $n\in \N$, we have
\[ C_{0}^{-1} \leq \frac{\mu_{s}([x]_n)}{e^{-nP(\Phi_{\A}^{s})}\|\A^n (x)\|^{s}}\leq C.\]
Note that if $\mu$ satisfies a Gibbs inequality with respect to $\Phi_{\A}^{s}$, then $\mu$ is fully supported on $\Si$.

Let $\s: \Si \to \Si$ be a left shift map. We say that an invariant measure $\mu \in \M(\Si, \s)$ is \textit{totally ergodic} if for every $n \in \N$, $\mu$ is ergodic with respect to $\s^n$.  We also define $\hat{\Si}:=\{1,2,\ldots,\ell\}^\Z$ equipped with a norm $d$, $\hat{\sigma}\left(\left(x_{k}\right)_{k \in \mathbb{Z}}\right):=\left(x_{k+1}\right)_{\mathbb{Z}}$ and $\M(\hat{\Si}, \hat{\sigma})$ the space of all $\hat{\sigma}$-invariant Borel probability measures on $\hat{\Si}$.

 We define the natural projection $\pi: \hat{\Si} \rightarrow \Si$ by $\pi\left(\left(x_{k}\right)_{k \in \mathbb{Z}}\right):=\left(x_{k}\right)_{k=0}^{\infty}$ that is continuous and surjective. Clearly $\hat{\mu} \mapsto \pi_{*} \hat{\mu}$ defines a continuous function $\M(\hat{\Si}, \hat{\sigma}) \rightarrow \M(\Si, \sigma)$ and since shift-invariant measures on $\Si$ and on $\hat{\Si}$ are both characterized by their values on cylinder sets this map is bijective. Let $\mu \in \M(\Si, \sigma)$, we will write $\hat{\mu}$ for the unique element of $\M(\hat{\Si}, \hat{\sigma})$ such that $\mu=\pi_{*} \hat{\mu}$, and we call $\hat{\mu}$ the natural extension of the measure $\mu$. Since properties such as mixing, ergodicity and total ergodicity can be characterized in terms of correlations between cylinder sets it is not difficult to see that each of those properties holds for an invariant measure $\mu \in \M(\Si, \sigma)$ if and only if the corresponding property holds for $\hat{\mu} \in \mathcal{M}(\hat{\Si}, \hat{\sigma})$.   We say that a measure $\hat{\mu}$ on $\hat{\Si}$ is a \textit{Bernoulli measure} if it has the form $\hat{\mu}=\left(\sum_{i=1}^{\ell} p_{i} \delta_{i}\right)^{\mathbb{Z}}$ for some probability vector $\left(p_{1}, \ldots, p_{\ell}\right)$ and that $\hat{\mu}$ has the \textit{Bernoulli property} if there exist a Bernoulli measure $\hat{v}$ on $\hat{\Si}$ and a measure-space isomorphism $\phi: \hat{\Si} \rightarrow \hat{\Si}$ such that $\phi \circ \hat{\sigma}=\hat{\sigma} \circ \phi$ and $\phi_{*} \hat{\mu}=\hat{v}$ (See \cite[Section 7]{Morris21} for more details). It is clear that every Bernoulli measure has the Bernoulli property, but the reverse is in general false. 

Morris \cite{Mor18, Mor19} showed that total
ergodicity of equilibrium states implies mixing and that the failure of mixing can
be characterized by certain structures of the cocycle. Recently, he \cite{Morris21} improved his own result
by showing that total ergodicity implies the Bernoulli property. Piraino \cite[Theorem 3.3]{Piraino20} showed that for any $s \geq 0$, 
a unique Gibbs state for $\A$, $\mu_s$, has the Bernoulli property when $\A$ is proximal and strongly irreducible.  By using Corollary \ref{thm-k-qm}, we generalize their results.
\begin{thm}\label{mixing and Bernoulli}
Let $\A\colon \Si \to \glr$ be a locally constant cocycle. Suppose that $\A^t$ and $\A^{\wedge m}$ are irreducible for every $t \mid d$ and $1\leq m \leq d-1$. Then for any $s>0$,
the unique Gibbs state $\mu_s$ for the norm potential $\Phi_{\A}^s$  is $\psi$-mixing:
$$
\lim _{n \rightarrow \infty} \sup _{I, J \in \L}\left|\frac{\mu_{s}\left([I] \cap \sigma^{-n-|I|}[J]\right)}{\mu_{s}([I]) \mu_{s}([J])}-1\right|=0,
$$
and its natural extension $\hat{\mu}$ has the Bernoulli property.
\end{thm}
\begin{proof}
Denote $\mu:=\mu_{s}.$ Since $\mu$ is the unique Gibbs equilibrium state for $\Phi_{\A}^{s}$,
there exists $C_0>0$ such that 
\begin{equation}\label{Gibbs-property}
C_{0}^{-1} \|\A_{I}\|^{s} \leq e^{|I| P\left(\Phi_{\A}^{s}\right)} \mu([I]) \leq C_{0} \|\A_{I}\|^{s}
\end{equation}
for every $I \in \L$. By Corollary \ref{thm-k-qm},  there exist an integer $m\in \N$ and constant $C_{1}>0$ such that for all $I, J \in \L$ there exists $K \in \L(m)$ such that
\begin{equation}\label{k-qm-property}
 \|\A_{IKJ}\| \geq C_1 \|\A_{I}\| \|\A_{J}\|.
\end{equation}
 Therefore, by \eqref{Gibbs-property} and \eqref{k-qm-property}, for every $I, J \in \L$ we have
$$\begin{aligned}
C_1 \mu([I]) \mu([J]) &\leq C_{0}^{2} C_{1} e^{-(|I|+|J|) P\left(\Phi_{\A}^{s}\right)}  \|\A_{I}\|^{s} \|\A_{J}\|^{s} \\
&\leq C_{0}^{2} e^{-(|I|+|J|) P\left(\Phi_{\A}^{s}\right)}   \|\A_{IKJ}\|^{s}  \\
&\leq C_{0}^{3} e^{|K| P\left(\Phi_{\A}^{s}\right)}  \mu([IKJ]) \\
&\leq C_{0}^{3} e^{m P\left(\Phi_{\A}^{s}\right)} \sum_{|K|=m} \mu([IKJ]) \\
& =C_{0}^{3} e^{m P\left(\Phi_{\A}^{s}\right)} \mu\left([I] \cap \sigma^{-m-|I|}[J]\right)
\end{aligned}
$$
so that
\begin{equation}\label{one-side-thm1}
\mu\left([I] \cap \sigma^{-m-|I|}[J]\right) \geq \kappa \mu([I]) \mu([J])
\end{equation}
where $\kappa:=C_{0}^{-3} C_{1} e^{-m P\left(\Phi_{\A}^{s}\right)}$. 

By \eqref{one-side-thm1}, for any $n \geq m$ we have that
$$
\begin{aligned}
\mu\left([I] \cap \sigma^{-n-|I|}[J]\right) &=\sum_{|K'|=n-m} \mu\left([IK'] \cap \sigma^{-m-|K'|-|I|}[ J]\right) \\
& \geq \kappa \sum_{|K'|=n-m} \mu([IK']) \mu([J]) \\
&=\kappa\mu([J]) \sum_{|K'|=n-m} \mu([IK']) \\
&=\kappa \mu([I]) \mu([J]) .
\end{aligned}
$$
Thus we have by an approximation argument that
\[\liminf _{n \rightarrow \infty} \mu\left(X \cap \sigma^{-n} Y\right) \geq \kappa \mu(X) \mu(Y)\]
for all $X, Y$ Borel measurable. The above inequality implies that $\mu$ is totally ergodic. Then, the proof follows from \cite[Theorem 1]{Morris21}.
\end{proof}

\subsection{Shrinking target and recurrence sets}

Let $T:X \to X$ be a topological dynamical system on a compact metric space $(X,d)$. Assume that $\mu$ is an $T$-invariant ergodic measure. By Birkhoff ergodic theorem,  for any ball $B$ in $X$ of positive $\mu$-measure, the set $$S:=\left\{x \in X: T^{n} x \in B, \text{ for infinitely many }  n \in \mathbb{N}\right\}$$
has full $\mu$-measure.

Now consider from the above definition of the set $S$ that both the center and the radius of the ball $B$ are allowed to vary in $n$;  given a function $h :\N \to \R_{+}$  tending to 0 as $n \rightarrow \infty$ and a sequence of points $\{z_n\}_{n\geq 1}$ in $X$, the set $S$ can be generalized to
$$
S(h)=\left\{x \in X: T^{n} x \in B\left(z_{n}, h(n)\right) \text {, for infinitely many } n \in \mathbb{N}\right\}.
$$
This set $S(h)$ is called the \textit{shrinking target}. Then one can ask how large the size of $S(h)$ is in the sense of measure and in the sense of dimension. This is called the \textit{shrinking target} problem by Hill and Velani \cite{HV} which concerns what happens if the target $B$ shrinks with time and more generally if the target also moves around with time. The points in $S(h)$ can be thought of as trajectories which hit a shrinking, moving target infinitely often. 

The shrinking target problem has intricate links to number theory when using naturally
arising sets in Diophantine approximation as the shrinking targets; e.g. see \cite{AB, BR, PS17}. 

The above works mostly concern conformal dynamics or dynamical systems in $\R^1$ and transitioning to higher dimensional non-conformal dynamics presents severe challenges. To overcome
the extreme challenges that affinities pose, a common approach is to “randomise” the affine maps
by considering typical translation parameter. Koivusalo and Ramirez \cite{KR} gave an expression for the Hausdorff dimension of a
self-affine shrinking target problem. They show that for a fixed symbolic target with exponentially
shrinking diameter and well-behaved affine maps, the Hausdorff dimension is typically given by
the zero of an appropriate pressure function. Strong assumptions are made on the affine system,
as well as the fixed target.

Let $\A= \left(A_{1}, A_{2}, \cdots, A_{\ell} \right)$ be a collection of non-singular $2 \times 2$ contracting matrices. Let $\mathbf{t}=\left\{t_{1}, t_{2}, \cdots, t_{\ell}\right\}$ be a collection of $k$ vectors in $\mathbb{R}^{2}$. Let $D_{\mathbf{t}}=\left\{f_{i}(x)=A_{i} x+t_{i}\right\}_{i=1}^{\ell}$  be an iterated function system formed by affine maps on $\R^{2}$.

 For a finite word $i \in \L$, $f_{i}=f_{i_{1}} \circ \cdots \circ f_{i_{n}}$. There exists a unique non-empty compact set $\Lambda \subset \mathbb{R}^{2}$ such that 
$$
\Lambda=\bigcup_{i=1}^{\ell} f_{i}(\Lambda);
$$
See \cite{Hut81}. Suppose that $\ell \geq 2$. Let us denote by $\pi_{\mathbf{t}}$ the natural projection from $\Si$ to the attractor of $\Lambda$, that is,
$$
\pi_{\mathbf{t}}(\mathbf{i})=\lim _{n \rightarrow \infty} f_{i_{1}} \circ \cdots \circ f_{i_{n}}(0)=\sum_{k=1}^{\infty} A_{i_{1}} \cdots A_{i_{k-1}} t_{i_{k}}
$$
Clearly, $\pi_{\mathbf{t}}(\mathbf{i})=f_{i_{1}}\left(\pi_{\mathbf{t}}(\sigma (\mathbf{i}))\right)$.

 We define the \textit{singular value function}  as follows
$$\varphi^{s}(A)= \begin{cases}\|A\|^{s}, & \text { if } 0 \leq s<1 \\ \|A\| \|A^{-1}\|^{-(s-1)}, & \text { if } 1 \leq s<2 \\ |\operatorname{det}(A)|^{s / 2}, & \text { if } 2 \leq s<\infty.\end{cases}$$

The pressure of the self-affine system is defined as
$$
P(\log \varphi^s(\mathcal{A}))=\lim _{n \rightarrow \infty} \frac{1}{n} \log \sum_{I \in \L(n)} \varphi^{s}\left(\A_{I}\right),
$$
whose existence of the limit is guaranteed from the submultiplicativity of $\varphi^{s}(A)$. For simplicity, we denote $P(s):= P(\log \varphi^s(\mathcal{A})).$

Let $\left(J_{k}\right)_{k \in \mathbb{N}} \in\left(\L \right)^{\mathbb{N}}$ be a sequence of target cylinders. We are interested in the shrinking target set
$$
S_{\mathbf{t}}\left(\left(J_{k}\right)_{k \in \mathbb{N}}\right)=\pi_{\mathbf{t}}\left\{\mathrm{i} \in \Si: \sigma^{k} \mathrm{i} \in\left[J_{k}\right] \text { for infinitely many } k \in \mathbb{N}\right\} .
$$

For our sequence of target cylinders, we define the following inverse lower pressure:
$$
\alpha(s):=\liminf_{k \rightarrow \infty} \frac{-1}{|J_{k}|} \log \varphi^{s}\left(\A_{J_{k}}\right)
$$
Let
$$
s_{0}:=\inf \{s>0: P(s) \leq \alpha(s)\} .
$$

Unfortunately, the
uncertainty of the length of the connecting word $K$ in the quasi-multiplicativity property does not let us to
study shrinking target and recurrence sets effectively. In order to study shrinking target and recurrence sets, B\'ar\'any and Troscheit \cite{barany2021dynamically} prove that $\A^{\wedge s}$ is uniform $k$-quasi-multiplicativity for every $s \in (0, d)$ when $\A$ is fully strongly irreducible and fully proximal. In the two dimensional case, we can improve their results \cite[Corollaries 2.7 and 2.8]{barany2021dynamically} by using Corollary \ref{thm-k-qm}.

\begin{lem}\label{singular value is QM}
Assume that $\A=(A_1, \ldots, A_{\ell})$ is a collection of non-singular $2 \times 2$ matrices. Suppose that $\A$ is $k$-quasi multiplicative. Then, for every $s\in [0, 2]$, there exist $k\in \N$ and $C>0$ such that for any $I,J \in \L$, there exists $K = K(I,J) \in \L(k)$ such that
\[
\varphi^{s}(\A_{IKJ}) \geq C \varphi^{s}(\A_{I}) \varphi^{s}(\A_{J}).
\]
\end{lem}
\begin{proof} 
Since $\A$ is is $k$-quasi multiplicative,  there exist $k\in \N$ and $c>0$ such that for any $I,J \in \L$, there exists $K \in \L(k)$ such that
\[
\|\A_{IKJ}\| \geq c\|\A_{I}\|\|\A_{J}\|.
\]
 When $s \in [0, 1]$, the proof follows by raising from $k$-quasi multiplicativity of $\A$ by power $s$.

  Now, we are going to prove it when $1<s \leq 2$. Notice that $\|A\|\|A^{-1}\|^{-(s-1)}=|\det(A)|^{s-1} \|A\|^{2-s}$ for any $A \in \text{GL}_{2}(\R)$. Then, the proof follows by multiplying, raising from $k$-quasi multiplicativity of $\A$ by power $2-s$, and raising from the determinant by power $s-1$.
\end{proof}

\begin{cor} Assume that $\A=(A_1, \ldots, A_{\ell})$ is a collection of $2 \times 2$ contracting matrices and $\left(J_{k}\right)_{k \in \mathbb{N}}$ is a sequence of target cylinders. Suppose that $\A$ and $\A^2$ are irreducible and $\|A_i\|<1 / 2$ for all $i\in \{1, \ldots, \ell \}$. Then
$$
\operatorname{dim}_{H} S_{\mathbf{t}}\left(\left(J_{k}\right)_{k}\right)=\min \left\{2, s_{0}\right\} \text { for Lebesgue-almost every } \mathbf{t} .
$$
Moreover, $\mathcal{L}_{2}\left(S_{\mathbf{t}}\left(\left(J_{k}\right)_{k}\right)\right)>0$ for Lebesgue-almost every $\mathbf{t}$ if $s_{0}>2$.
\end{cor}
\begin{proof}

By Corollary \ref{thm-k-qm}, $\A$ is $k$-quasi multiplicative. Then,  the statement follows from the combination of \cite[Theorem 2.2]{barany2021dynamically} and Lemma \ref{singular value is QM}.
\end{proof}

Suppose that $\psi: \mathbb{N} \mapsto \mathbb{N}$, and $\beta:=\liminf_{n \rightarrow \infty} \frac{\psi(n)}{n}$. We consider the recurrent set
$$
R_{\mathbf{t}}(\psi):=\pi_{\mathbf{t}}\left\{\mathrm{i} \in \Si: \sigma^{k} \mathrm{i} \in\left[\left.\mathrm{i}\right|_{\psi(k)}\right] \text { for infinitely many } k \in \mathbb{N}\right\} \text {. }
$$
We define the square-pressure function
$$
P_{2}(s):=\lim _{n \rightarrow \infty} \frac{-1}{n} \log \sum_{i \in \L(n)}\left(\varphi^{s}\left(\A_{i}\right)\right)^{2},
$$
where the limit exists because of the subadditivity of $\varphi^{s}(\A)$. Moreover, the pressure is continuous in $s$, strictly increasing, and satisfies $P_{2}(0)=-\log N$ and $P_{2}(s) \rightarrow \infty$ as $s \rightarrow \infty$.
\begin{cor}
Assume that $\A=(A_1, \ldots, A_{\ell})$ is a collection of $2 \times 2$ contracting matrices. Suppose that $\A$ and $\A^2$ are irreducible and $\|A_i\|<1 / 2$ for all $i\in \{1, \ldots, \ell \}$. Let  $\psi :\N \to \N$ with $\beta=\liminf_{n \to \infty} \frac{\psi(n)}{n}<1$, then
$$
\operatorname{dim}_{H} R_{\mathbf{t}}\left(\psi \right)=\min \left\{2, r_{0}\right\} \text { for Lebesgue-almost every } \mathbf{t},
$$
where $r_0$ is the unique solution of the equation
\[(1-\beta) P\left(r_{0}\right)=\beta P_{2}\left(r_{0}\right).\]
Moreover, $\mathcal{L}_{2}\left(R_{\mathbf{t}}\left(\left(\psi \right)_{k}\right)\right)>0$ for Lebesgue-almost every $\mathbf{t}$ if $r_{0}>2$.
\end{cor}
\begin{proof}
By Corollary \ref{thm-k-qm}, $\A$ is $k$-quasi multiplicative. Thus,  the proof follows from the combination of \cite[Theorem 2.4]{barany2021dynamically} and Lemma \ref{singular value is QM}.
\end{proof}

\bibliographystyle{acm}
\bibliography{fl_spannability}
\end{document}